\documentclass[11pt,reqno]{amsproc}
 \usepackage{hyperref}
\newtheorem{theorem}{\bf{Theorem}}[section] 
\newtheorem{lemma}[theorem]{\bf{Lemma}}     

\newtheorem{corollary}[theorem]{\bf{Corollary}}

\newtheorem{definition}[theorem]{\bf{Definition}}

\title[Characterization of Symmetric Amenability of Unital Banach Algebras]
{Characterization of Symmetric Amenability of Unital Banach Algebras} 
\author{Ali Jabbari}
\address{Young Researchers and Elite Club, Ardabil, Iran}
\email{jabbari\underline{ }al@yahoo.com}
\author[E. Ebadian]{Ali Ebadian}
\address{
Faculty of Science, Department of Mathematics, Urmia University, Urmia, Iran}
\email{ebadian.ali@gmail.com}

\subjclass[2010]{Primary 46H05, Secondary  43A10}
\keywords{amenability, Banach algebra, symmetrically amenable}

\begin{document}
\maketitle

\begin{abstract}
In this paper, we introduce $p$-amenability,  bounded $s$-symmetric approximate and $s$-symmetric virtual diagonals for a Banach algebra $\mathfrak{A}$ where $s$ is a non-zero element of algebraic center of $\mathfrak{A}$ that is denoted  by $Z(\mathfrak{A})$. We show that if a Banach algebra $\mathfrak{A}$ is $p$-amenable then it has bounded $s$-symmetric approximate and $s$-symmetric virtual diagonals and by this fact we prove that if the Banach algebra $\mathfrak{A}$ is unital then $p$-amenability and symmetric amenability are equivalent.
\end{abstract}


\section{Introduction} 
Let $\mathfrak{A}$ be a Banach algebra and let $X$ be a Banach $\mathfrak{A}$-bimodule. A derivation is a linear map $D:\mathfrak{A}\longrightarrow X$ such that $D(ab)=a\cdot D(b)+a\cdot D(b)$ for every $a,b\in \mathfrak{A}$. A derivation $D:\mathfrak{A}\longrightarrow X$ is called inner if there is a $x\in X$ such that $D(a)=a\cdot x-x\cdot a$ for every $a\in \mathfrak{A}$. Note that, if $X$ is a Banach $\mathfrak{A}$-bimodule, then $X^*$ becomes a  Banach $\mathfrak{A}$-bimodule by the following actions
$$\langle x,a\cdot f\rangle=\langle x\cdot a, f\rangle,\ \ \ \  \langle x,f\cdot a\rangle=\langle a\cdot x,f\rangle,$$
for every $f\in X^*$, $x\in X$ and $a\in \mathfrak{A}$.

A Banach algebra $\mathfrak{A}$ is amenable if every bonded derivation from $\mathfrak{A}$ into dual of any Banach $\mathfrak{A}$-bimodule $X^*$ is inner.

Let $\mathfrak{A}$ be a Banach algebra, Johnson has shown that $\mathfrak{A}$ is amenable if and only if it has a bounded approximate diagonal \cite{joh}; that is a bounded net $(m_\alpha)_\alpha$ in $\mathfrak{A}\widehat{\otimes}\mathfrak{A}$ with
\begin{enumerate}
  \item $a\cdot m_\alpha-m_\alpha\cdot a\longrightarrow0$,
  \item $a\cdot\pi(m_\alpha)\longrightarrow a$,
\end{enumerate}
for every $a\in\mathfrak{A}$, where $\pi:\mathfrak{A}\widehat{\otimes}\mathfrak{A}\longrightarrow\mathfrak{A}$ is defined by $\pi(a\otimes b)=ab$ for every $a,b\in\mathfrak{A}$. An element $t\in\mathfrak{A}\widehat{\otimes}\mathfrak{A}$ is called symmetric if $t^\circ=t$, where  $``\circ"$ is the flip map is defined by $(a\otimes b)^\circ =b\otimes a$. Johnson in \cite{joh1} is introduced symmetric amenability of Banach algebras. He called a Banach algebra $\mathfrak{A}$ is symmetrically amenable if it has a bounded approximate identity consisting of symmetric elements. Consider the opposite algebra $\mathfrak{A}^\circ$ that is the Banach space $\mathfrak{A}$ with product $a\circ b=ba$. A bounded approximate diagonal  in $\mathfrak{A}\widehat{\otimes}\mathfrak{A}$ for $\mathfrak{A}^\circ$ is a bounded net $(m_\alpha)_\alpha$  in $\mathfrak{A}\widehat{\otimes}\mathfrak{A}$ if
\begin{enumerate}
  \item[(1)$^\circ$] $a\circ m_\alpha-m_\alpha\circ a\longrightarrow0$,
   \item[(2)$^\circ$] $a\circ\pi^\circ(m_\alpha)\longrightarrow a$,
\end{enumerate}
for every $a\in\mathfrak{A}$. The Banach algebra $\mathfrak{A}$ is symmetrically amenable if and only if there is a bounded net  $(m_\alpha)_\alpha$ in $\mathfrak{A}\widehat{\otimes}\mathfrak{A}$ such that satisfies (1), (2), (1)$^\circ$ and (2)$^\circ$ \cite[Proposition 2.2]{joh1}. For unital Banach algebras, if the conditions (1), (2),  and (2)$^\circ$ hold, then it is symmetrically amenable \cite[Proposition 2.6]{joh1}.

In the next section, we introduce two new concepts related to inner derivations on Banach algebras that we call them $p$ and $p^\circ$-inner derivations. By these new defined derivations we introduce $p$-amenability of Banach algebras.

The final section considers symmetrically amenable unital Banach algebras. In this section, we prove that symmetric amenability and $p$-amenability of Banach algebras are equivalent. Moreover, we characterize symmetric amenability of Banach algebras in the sense of Lau's paper \cite{lau}. Note that symmetric amenability of Banach algebras is investigated by the existence of symmetric approximate diagonal and in the section 3, we consider it with derivations.
\section{$p$-Amenability}
In this section, we consider derivations from a Banach algebra into tensor product of its and a Banach $\mathfrak{A}$-bimodule. Let $\mathfrak{A}$ be a Banach algebra and $X$ be a Banach $\mathfrak{A}$-bimodule  for $n\in\mathbb{N}$ and by $\mathfrak{A}^{(n)}$ and $X^{(n)}$, we mean the $n$-th duals of $\mathfrak{A}$ and $X$, respectively. We define $p:\mathfrak{A}\widehat{\otimes}X^{*}\longrightarrow X^*$ such that $p({a}\otimes\mathfrak{x})={a}\cdot\mathfrak{x}$ and $p^\circ({a}\otimes\mathfrak{x})=\mathfrak{x}\cdot a$ for every $\mathfrak{a}\in\mathfrak{A}$ and $\mathfrak{x}\in X^{*}$. Now; we define $p$-inner derivations and $p$-amenability as follows:
\begin{definition}\label{d1}
Let $\mathfrak{A}$ be a Banach algebra and $X$ be a Banach $\mathfrak{A}$-bimodule. We say that a derivation $D:\mathfrak{A}\longrightarrow X^*$ is $p$-inner, if there exists an element  $t\in \mathfrak{A}\widehat{\otimes} X^*$ such that
\begin{equation*}
  D(a)=a\cdot p(t)-p(t)\cdot a,
\end{equation*}
for all $a\in \mathfrak{A}$. Similarly,  a derivation $D:\mathfrak{A}^\circ\longrightarrow X^*$ is $p^\circ$-inner, if there exists an element  $t'\in \mathfrak{A}\otimes X^*$ such that
\begin{equation*}
  D(a)=a\cdot p^\circ(t)-p^\circ(t)\cdot a,
\end{equation*}
for all $a\in \mathfrak{A}$.
\end{definition}

\begin{definition}
  Let $\mathfrak{A}$ be a Banach algebra and $X$ be a Banach $\mathfrak{A}$-bimodule. We say that $\mathfrak{A}$ is $p$-amenable if every derivation from $\mathfrak{A}$ into $X^*$  and every derivation from $\mathfrak{A}^\circ$ into $X^*$ are $p$ and $p^\circ$-inner, respectively.
\end{definition}

In light of Proposition 2.1.5 of \cite{runde} we have the following:
\begin{lemma}\label{l1}
 Let $\mathfrak{A}$ be a Banach algebra with a bounded approximate identity. Then the following assertions are equivalent:
 \begin{itemize}
   \item[(i)] Every derivation from $\mathfrak{A}\ (\mathfrak{A}^\circ)$   into $X^*$ is $p$-inner ($p^\circ$-inner), where $X$ is an arbitrary  Banach $\mathfrak{A}$-bimodule.
   \item[(ii)] Every derivation from $\mathfrak{A}\ (\mathfrak{A}^\circ)$ into $X^*$ is $p$-inner ($p^\circ$-inner), where $X$ is an arbitrary  pseudo-unital Banach $\mathfrak{A}$-bimodule.
 \end{itemize}
\end{lemma}

Similar to amenable Banach algebras, we have the following result:
\begin{lemma}\label{p1}
  Every  $p$-amenable Banach algebra has a bounded approximate identity.
\end{lemma}
\begin{proof}
The proof is similar to the proof of Proposition 2.2.1 of \cite{runde}. Let $\mathfrak{B}$ be the Banach $\mathfrak{A}$-bimodule with underlying space $\mathfrak{A}$ with the following module action:
 $$a\cdot x=ax,\hspace{1cm}x\cdot a=0,$$
 and
  $$a\circ x=0,\hspace{1cm}x\circ a=ax,$$
 for every $a\in\mathfrak{A}$ and $x\in \mathfrak{B}$.
  We work on $\mathfrak{A}$ and for the case $\mathfrak{A}^\circ$ proof is similar. Let $D:\mathfrak{A}\longrightarrow\mathfrak{B}^{**}$ be the canonical embedding of $\mathfrak{A}$ into it's second dual. Clearly, $D\in \mathcal{Z}^1(\mathfrak{A},\mathfrak{B}^{**})$. Thus, there is a $t\in \mathfrak{A}\otimes\mathfrak{B}^{**}$ such that $a=a\cdot p(t)$ for all $a\in\mathfrak{A}$. Thus, there exists a bounded net $(e_\alpha)$ in $\mathfrak{A}$ such that $e_\alpha\stackrel{w^*}{\longrightarrow}p(t)$. This follows that $a=w-\lim_\alpha ae_\alpha$ for all $a\in\mathfrak{A}$. Passing to convex combinations, implies that $(e_\alpha)$ is  a bounded right approximate identity for $\mathfrak{A}$ and a bounded left approximate identity for $\mathfrak{A}^\circ$.

 Similarly, one can show that there is a bounded left approximate identity and bounded right approximate identity $(f_\beta)$ for  $\mathfrak{A}$ and $\mathfrak{A}^\circ$, respectively. Now; by setting $E_{\alpha,\beta}=e_\alpha+f_\beta-e_\alpha f_\beta$, we obtain a bounded  approximate identity  for  $\mathfrak{A}$ and $\mathfrak{A}^\circ$.
\end{proof}

\section{Characterization of Symmetric Amenability}

In this section, we give our main results related to the symmetric amenability of Banach algebras. In the other word, we characterize symmetric amenability similar to amenability, by derivations. We start off with the following definitions:
\begin{definition}\label{def1}
  Let $\mathfrak{A}$ be a Banach algebra with $0\neq s\in Z(\mathfrak{A})$.
  \begin{itemize}
    \item[(i)]  We call an element $M\in \mathfrak{A}^{**}\widehat{\otimes}\mathfrak{A}^{**}$  a $s$-symmetric virtual diagonal for $\mathfrak{A}$ if for every $a\in \mathfrak{A}$, we have
  \begin{enumerate}
    \item $a\cdot M=M\cdot a$ and $a\cdot \pi^{**}(M)=as^2$,
    \item  $a\circ M=M\circ a$ and $a\circ \pi^{\circ**}(M)=as^2$.
  \end{enumerate}
    \item[(ii)]  We call a bounded net $(m_\alpha)_\alpha$ in $\mathfrak{A}\widehat{\otimes}\mathfrak{A}$ a $s$-symmetric approximate diagonal for $\mathfrak{A}$ if for every $a\in \mathfrak{A}$, we have
  \begin{enumerate}
    \item $a\cdot m_\alpha- m_\alpha\cdot a\longrightarrow0$ and $a\cdot \pi(m_\alpha)\longrightarrow as^2$,
    \item  $a\circ m_\alpha-m_\alpha\circ a\longrightarrow0$ and $a\circ \pi^\circ(m_\alpha)\longrightarrow as^2$.
  \end{enumerate}
   \item[(iii)] We say that the Banach algebra $\mathfrak{A}$ is $s$-symmetrically amenable if it has a bounded $s$-symmetric approximate diagonal.
  \end{itemize}
\end{definition}

The following result is very important for characterization of symmetric amenability.
\begin{theorem}\label{t1}
  Let $\mathfrak{A}$ be Banach algebra with $0\neq s\in Z(\mathfrak{A})$, if $\mathfrak{A}$ is $p$-amenable, then the following assertions hold and are equivalent:
  \begin{itemize}
    \item[(i)]  $\mathfrak{A}$ has a $s$-symmetric virtual diagonal;
    \item[(ii)] $\mathfrak{A}$ is $s$-symmetrically amenable.
  \end{itemize}
\end{theorem}
\begin{proof}
(i)$\longrightarrow$(ii) Assume that $\mathfrak{A}$ is $p$-amenable. Then by Lemma \ref{p1}, there is a bounded approximate identity $(e_\alpha)$ for $\mathfrak{A}$ such that $e_\alpha\stackrel{w^*}{\longrightarrow}e\in \mathfrak{A}^{**}$. Denote $e\otimes e$ by $E$ that is in $\mathfrak{A}^{**}\widehat{\otimes}\mathfrak{A}^{**}\subseteq(\mathfrak{A}\widehat{\otimes}\mathfrak{A})^{**}$. Now; consider $\pi^{**}:(\mathfrak{A}\widehat{\otimes}\mathfrak{A})^{**}\longrightarrow \mathfrak{A}^{**}$ and $\pi^{\circ **}:(\mathfrak{A}\widehat{\otimes}\mathfrak{A})^{**}\longrightarrow \mathfrak{A}^{**}$. Then
\begin{equation}\label{e1t1}
\pi^{**}(\delta_E(a))= \pi^{**}(a\cdot E-E\cdot a)=w^*-\lim_\alpha\pi(ae_\alpha\otimes e_\alpha-e_\alpha\otimes ae_\alpha)=0,
\end{equation}
and
\begin{equation}\label{e2t1}
\pi^{\circ **}(\delta_E^\circ(a))=  \pi^{\circ **}(a\circ E-E\circ a)=w^*-\lim_\alpha\pi^\circ (e_\alpha\otimes ae_\alpha-e_\alpha a\otimes e_\alpha)=0,
\end{equation}
for every $a\in \mathfrak{A}$. Then (\ref{e1t1}) implies that $\delta_E\subseteq \ker\pi^{**}=(\ker\pi)^{**}$. Moreover, if we replace $\pi$ by $\pi^\circ$ in \eqref{e1t1}, we have $\delta_E(a)\subseteq \ker\pi^{\circ**}=(\ker\pi^\circ)^{**}$. Also, (\ref{e2t1}) shows that $\delta_E^\circ\subseteq \ker\pi^{\circ**}=(\ker\pi^\circ)^{**}$ and similarly, $\delta_E^\circ\subseteq\ker\pi^{**}=(\ker\pi)^{**}$. These facts together imply that $\ker\pi^{**}\cap\ker\pi^{\circ**}\neq\emptyset$. Since  $(\ker\pi)^{**}$ and $(\ker\pi)^{\circ**}$ are Banach $\mathfrak{A}$-bimodules, $(\ker\pi)^{**}\cap(\ker\pi)^{\circ**}=K$ is a Banach $\mathfrak{A}$-bimodule. Thus, there exist $t,t'\in \mathfrak{A}\otimes K$ such that $\delta_E=\delta_{p(t)}$ and $\delta_E^\circ=\delta_{p^\circ(t')}$, where $p:\mathfrak{A}\otimes K\longrightarrow K$.  Define
$$M=\frac{1}{2}[s\cdot(E-p(t))\cdot s+s\circ(E-p^\circ(t'))\circ s].$$

Then $M$ satisfies in conditions (1) and (2) in case (i) of Definition \ref{def1}.

(ii)$\longrightarrow$(i) Suppose that $M\in \mathfrak{A}^{**}\widehat{\otimes}\mathfrak{A}^{**}$ is a $s$-symmetric virtual diagonal for $\mathfrak{A}$, then clearly,  there is a net $(m_\alpha)_\alpha$ in $\mathfrak{A}\widehat{\otimes}\mathfrak{A}$ such that  $M=w^*-\lim m_\alpha$. It easy to see that $(m_\alpha)_\alpha$ is a $s$-symmetric approximate diagonal for $\mathfrak{A}$.
\end{proof}
\begin{theorem}\label{tt1}
  Let $\mathfrak{A}$ be a unital Banach algebra, then the following assertions equivalent:
  \begin{itemize}
    \item[(i)]  $\mathfrak{A}$ is $p$-amenable
    \item[(ii)] $\mathfrak{A}$ has a symmetric virtual diagonal;
    \item[(iii)] $\mathfrak{A}$ is symmetrically amenable.
  \end{itemize}
\end{theorem}
\begin{proof}
The cases (i)$\longrightarrow$(ii)$\longrightarrow$(iii) are proved in Theorem \ref{t1}. Thus, we just prove the case (iii)$\longrightarrow$(i). Assume that $\mathfrak{A}$ is symmetrically amenable, hence, there is a bounded symmetric approximate diagonal $(m_\alpha)_\alpha$ in $\mathfrak{A}\widehat{\otimes}\mathfrak{A}$ such that satisfies in conditions (1), (2), (1)$^\circ$ and (2)$^\circ$. By Lemma 3 of \cite[Chapter VI]{bd}, we can write
$$m_\alpha=\sum_{n=1}^\infty a_n^\alpha\otimes b_n^\alpha,\ \ \sum_{n=1}^\infty\|a_n^\alpha\|\|b_n^\alpha\|<\infty.$$

 Similar to \cite[Theorem 2.2.4]{runde}, let $D:\mathfrak{A}\longrightarrow X^*$ be a continuous derivation, where $X$ is a pseudo-unital Banach $\mathfrak{A}$-bimodule (Lemma \ref{l1}). Therefore, $\sum_{n=1}^\infty p(a_n^\alpha\otimes D(b_n^\alpha))=\sum_{n=1}^\infty a_n^\alpha\cdot D(b_n^\alpha)$ is a bounded net in $X^*$. This implies that it has a $w^*$-accumulation point such as $\mathfrak{x}\in X^*$. Without loss of generality, suppose that  $\mathfrak{x}=w^*-\lim_\alpha \sum_{n=1}^\infty a_n^\alpha\cdot D(b_n^\alpha)$. Since $\mathfrak{A}$ is unital, Cohen factorization Theorem says that there exist $b\in \mathfrak{A}$ and $y\in X^*$ such that $ky=\mathfrak{x}=p(k\otimes y)$. Set $t=k\otimes y$, then
 $$D(a)=a\cdot p(t)-p(t)\cdot a,$$
  for every $a\in \mathfrak{A}$. Similarly, one can show that for every derivation $D:\mathfrak{A}^\circ \longrightarrow X^*$ there is a $t'\in \mathfrak{A}\widehat{\otimes} X^*$ such that  $$D(a)=a\cdot p^\circ(t')-p^\circ(t')\cdot a,$$
  for every $a\in \mathfrak{A}$. Hence, $\mathfrak{A}$ is $p$-amenable.
\end{proof}

Johnson proved that the group algebra $L^1(G)$  defined on a locally compact group $G$ is  amenable if and only if it is symmetrically amenable \cite[Theorem 4.1]{joh1}. Moreover, Blanco, showed that the Banach algebra $\mathcal{A}(X)$ of approximable  operators acting on a Banach space $X$ is amenable if and only if it is symmetrically amenable \cite[Theorem 3.1]{bl}. Now, by the following, we show that symmetric amenability and amenability are equivalent on $M(G)$ the Banach algebra of complex Borel measures on a locally compact group $G$.
\begin{corollary}
Let $G$ be a locally compact group. Then $M(G)$ is symmetrically amenable if and only if $G$ is discrete and amenable.
\end{corollary}
\begin{proof}
Let $M(G)$ be symmetrically amenable. By Theorem \ref{tt1}, $M(G)$ has a symmetric virtual diagonal and Theorem \ref{t1} implies that it has a bounded symmetric approximate diagonal. Thus, $M(G)$ is amenable, because of that it has a virtual diagonal \cite[Lemma 2.1]{joh}. This implies that $G$ is discrete and amenable \cite[Theorem 1.1]{dgh}.

Conversely, assume that $G$ is discrete and amenable. Then $M(G)=\ell^1(G)$ and $\ell^1(G)$ is amenable by Johnson Theorem. Now, apply Theorem 4.1 of \cite{joh1}.
\end{proof}

Let $\mathfrak{A}$ be a Banach algebra and $X$ be a Banach $\mathfrak{A}$-bimodule and  $\mathfrak{A}^\circ$-bimodule with defining module actions are defined in the previous section. Following Lau \cite{lau}, we define
$$Z(\mathfrak{A}, X^*)=\bigcap_{a\in\mathfrak{A}}\left\{f\in X^*:a\cdot f=f\cdot a\right\}.$$

The following result is the Theorem 1 of \cite{lau} in the sense of symmetric amenability that we give its proof where it is different from the proof of the mentioned Theorem.

\begin{theorem}
Let $\mathfrak{A}$ be a unital Banach algebra with unit $e_\mathfrak{A}$ and $X$ be a Banach $\mathfrak{A}$-bimodule.  Then the following statements are equivalent:
\begin{itemize}
  \item[(i)] $\mathfrak{A}$ is symmetrically amenable;
  \item[(ii)] For any Banach $\mathfrak{A}$-submodule $Y$ of $X$, each linear functional in $Z(\mathfrak{A}, Y^*)$  has an extension to a linear functional in $Z(\mathfrak{A}, X^*)$;
  \item[(iii)] There is a bounded projection from $X^*$ onto $Z(\mathfrak{A}, X^*)$ which commutes with any bounded linear operator from $X^*$ into $X^*$ commuting with the action of $\mathfrak{A}$ on $X$.
\end{itemize}
\end{theorem}
\begin{proof}
(i)$\longrightarrow$(ii) The quotient Banach space $X/Y$ is a Banach $\mathfrak{A}$-bimodule   by the following actions
\begin{equation}\label{1}
    a\cdot (x+Y)=a\cdot x+Y, \hspace{1cm}(x+Y)\cdot a=x\cdot a+Y,
\end{equation}
for every $a\in\mathfrak{A}$ and $x\in X$. Let  $\alpha\in Z(\mathfrak{A}, Y^*)$. Suppose that $\overline{\alpha}\in X^*$ is an extension of $\alpha$.  Then for every $a\in\mathfrak{A}$ and $y\in Y$ we have
\begin{eqnarray}
\nonumber
 \langle y, a\cdot\overline{\alpha}\rangle-\langle y,\overline{\alpha}\cdot a\rangle  &=& \langle y\cdot a, \overline{\alpha}\rangle-\langle a\cdot y, \overline{\alpha}\rangle\\
   \nonumber
   &=&  \langle y,a\cdot  \overline{\alpha}- \overline{\alpha}\cdot a\rangle=0.
\end{eqnarray}

This means that  $a\cdot  \overline{\alpha}- \overline{\alpha}\cdot a\in Y^\bot=\{x^*\in X^*|~ \langle y, x^*\rangle=0,~ \emph{\emph{for}}~ \emph{\emph{every}}~ y\in Y\}$. Clearly, $\mathbf{q}:Y^\bot\longrightarrow (X/Y)^*$ is an $\mathfrak{A}$-module  isometry and surjective. Define $D_1:\mathfrak{A}\longrightarrow (X/Y)^*$ by $D_1(a)=\mathbf{q}(a\cdot \overline{\alpha}- \overline{\alpha}\cdot a)$ for every $a\in \mathfrak{A}$. It is clear that $D_1$ is a bounded derivation.

Symmetric amenability of $\mathfrak{A}$ implies that it is $p$-amenable by Theorem \ref{t1}. Thus, there is a $t\in \mathfrak{A}\widehat{\otimes}(X/Y)^*$ such that $D_1(a)=a\cdot p^\circ(t)-p^\circ(t)\cdot a$ for every $a\in \mathfrak{A}^\circ$. Then there exists $\beta\in Y^\bot$ ($\mathbf{q}$ is surjective) such that $D_1(a)=a\cdot\mathbf{q}(\beta)-\mathbf{q}(\beta)\cdot a$, for all $a\in \mathfrak{A}$. Similarly, by defining $D_2:\mathfrak{A}^\circ\longrightarrow (X/Y)^*$ by $D_2(a)=\mathbf{q}(a\cdot  \overline{\alpha}- \overline{\alpha}\cdot a)$
for every $a\in \mathfrak{A}$,   one can find a $\lambda\in Y^\bot$
such that $D_2(a)=a\cdot\mathbf{q}(\lambda)-\mathbf{q}(\lambda)\cdot a$, for all $a\in \mathfrak{A}$.
 Now, set $\gamma=\overline{\alpha}-\beta-\lambda$. Then
\begin{equation}\label{}
 \nonumber
 \langle y,a\cdot (\overline{\alpha}-\beta-\lambda)-(\overline{\alpha}-\beta-\lambda)\cdot a\rangle=0,
 \end{equation}
for every $a\in\mathfrak{A}$ and $y\in Y$.
This means that $\gamma\in  Z(\mathfrak{A}, X^*)$.
Therefore (i) implies (ii).

(ii)$\longrightarrow$(iii) This case is exactly similar to the (ii)$\longrightarrow$(iii) of \cite[Theorem 1]{lau}.

(iii)$\longrightarrow$(i) Set $X=\mathfrak{A}\widehat{\otimes}\mathfrak{A}$, which $X$ is a Banach $\mathfrak{A}$-bimodule and $\mathfrak{A}^\circ$-bimodule by the following actions
\begin{equation}\label{5}
    a\cdot(b\otimes c)=ab\otimes c,\hspace{1cm}(b\otimes c)\cdot a=b\otimes ca,
\end{equation}
and
\begin{equation}\label{55}
    a\circ(b\otimes c)=b\otimes ac\hspace{1cm}(b\otimes c)\circ a=ba\otimes c,
\end{equation}
for all $a,b,c\in\mathfrak{A}$. Let $\mathcal{F}=\{R_a~|~a\in\mathfrak{A}\}\cup \{L_a~|~a\in\mathfrak{A}\}$ be a family of bounded linear operators from $X$ into $X$, such that
\begin{equation}\label{5}
    L_a(b\otimes c)=b\otimes ac\hspace{1cm}\emph{\emph{and}}\hspace{1cm}R_a(b\otimes c)=ba\otimes c,
\end{equation}
for all $a,b,c\in\mathfrak{A}$. Then
  every member of $\mathcal{F}$ commutes with the actions of $\mathfrak{A}$ on $X$. So, by (iii), there exists a bounded surjective projection $P:X^*\longrightarrow X^*$ such that $PT^*=T^*P$ for all $T\in\mathcal{F}$. Define $q:X^*\longrightarrow X^*$ by $\langle a\otimes b, q(f)\rangle=\langle b\otimes a, f\rangle$ for all  $a,b\in\mathfrak{A}$ and all $f\in X^*$. Set $M=q^*(P^*(e_\mathfrak{A}\otimes e_\mathfrak{A}))$. We claim that $M$ is a bounded symmetric virtual diagonal. According to definition of $q$ we have
\begin{eqnarray}\label{6}
 \nonumber
  \langle c\otimes d, q(x^*\cdot a)\rangle &=& \langle d\otimes c, x^*\cdot a\rangle=\langle ad\otimes c, x^*\rangle =\langle c\otimes ad, q(x^*)\rangle \\
     &=& \langle L_a(c\otimes d), q(x^*)\rangle =\langle c\otimes d, L_a^*q(x^*)\rangle,
\end{eqnarray}
\begin{eqnarray}\label{7}
 \nonumber
  \langle c\otimes d, q(a\cdot x^*)\rangle &=& \langle d\otimes c, a\cdot x^*\rangle=\langle  d\otimes ca, x^*\rangle =\langle ca\otimes d, q(x^*)\rangle \\
      &=& \langle R_a(c\otimes d), q(x^*)\rangle =\langle c\otimes d, R_a^*q(x^*)\rangle,
\end{eqnarray}
\begin{eqnarray}\label{66}
 \nonumber
  \langle c\otimes d, q(x^*\circ a)\rangle &=& \langle d\otimes c, x^*\circ a\rangle=\langle a\circ (d\otimes c), x^*\rangle =\langle ac\otimes d, q(x^*)\rangle \\
     &=& \langle a\cdot(c\otimes d), q(x^*)\rangle =\langle c\otimes d, q(x^*)\cdot a\rangle,
\end{eqnarray}
and
\begin{eqnarray}\label{77}
 \nonumber
  \langle c\otimes d, q(a\circ x^*)\rangle &=& \langle d\otimes c, a\circ x^*\rangle=\langle  (d\otimes c)\circ a, x^*\rangle =\langle c\otimes da, q(x^*)\rangle \\
      &=& \langle (c\otimes d)\cdot a, q(x^*)\rangle =\langle c\otimes d, a\cdot q(x^*)\rangle,
\end{eqnarray}
for all $a,b,c,d\in\mathfrak{A}$ and  $x^*\in X^*$. Then by properties of $P$ and relations (\ref{6}) and (\ref{7}),  we have
\begin{equation}\label{eqf1}
   \langle x^*, M\cdot a\rangle=\langle  Pq(x^*)\cdot a,e_\mathfrak{A}\otimes e_\mathfrak{A}\rangle\ \text{and}\ \langle x^*,a\cdot M\rangle =\langle a\cdot Pq(x^*),e_\mathfrak{A}\otimes e_\mathfrak{A}\rangle,
\end{equation}
for all $x^*\in X^*$. Since $Pq(x^*)$ is in $Z(\mathfrak{A},X^*)$, \eqref{eqf1} implies that $a\cdot M=M\cdot a$ for every $a\in\mathfrak{A}$. Similar to \cite{lau}, we have $a\cdot \pi^{**}(M)=a$ for every $a\in\mathfrak{A}$. Regarding as $P$ is an $\mathfrak{A}$-module morphism and relations \eqref{66} and  \eqref{77}, we have
\begin{eqnarray}\label{8}
\nonumber
 \langle x^*, M\circ a\rangle  &=& \langle a\circ x^*, q^*(P^*(e_\mathfrak{A}\otimes e_\mathfrak{A}))\rangle=\langle q(a\circ x^*), P^*(e_\mathfrak{A}\otimes e_\mathfrak{A})\rangle \\
 \nonumber
     &=&  \langle a\cdot q(x^*), P^*(e_\mathfrak{A}\otimes e_\mathfrak{A})\rangle= \langle P(a\cdot q(x^*)),e_\mathfrak{A}\otimes e_\mathfrak{A}\rangle\\
        &=&  \langle a\cdot Pq(x^*),e_\mathfrak{A}\otimes e_\mathfrak{A}\rangle
\end{eqnarray}
and
\begin{eqnarray}\label{9}
\nonumber
 \langle x^*,a\circ M\rangle  &=& \langle  x^*\circ a, q^*(P^*(e_\mathfrak{A}\otimes e_\mathfrak{A}))\rangle=\langle q(x^*\circ a), P^*(e_\mathfrak{A}\otimes e_\mathfrak{A})\rangle \\
 \nonumber
     &=&  \langle q(x^*)\cdot a, P^*(e_\mathfrak{A}\otimes e_\mathfrak{A})\rangle= \langle P(q(x^*)\cdot a),e_\mathfrak{A}\otimes e_\mathfrak{A}\rangle\\
       &=&  \langle Pq(x^*)\cdot a,e_\mathfrak{A}\otimes e_\mathfrak{A}\rangle
\end{eqnarray}
for all $x^*\in X^*$.
Therefore (\ref{8}) and (\ref{9}) imply that $a\circ M=M\circ a$ for every $a\in\mathfrak{A}$. Moreover,
\begin{eqnarray*}
   \langle f, a\circ\pi^{\circ**}M\rangle&=&  \langle f\circ a, \pi^{\circ**}M\rangle=\langle q(\pi^{\circ*}(f\circ a)),P^*(e_\mathfrak{A}\otimes e_\mathfrak{A})\rangle\\
   &=&  \langle P(q(\pi^{\circ*}(f\circ a))), e_\mathfrak{A}\otimes e_\mathfrak{A}\rangle=\langle q(\pi^{\circ*}(f\circ a)), e_\mathfrak{A}\otimes e_\mathfrak{A}\rangle\\
   &=&  \langle \pi^{\circ*}(f\circ a), e_\mathfrak{A}\otimes e_\mathfrak{A}\rangle=\langle f\circ a, \pi^{\circ**}(e_\mathfrak{A}\otimes e_\mathfrak{A})\rangle\\
   &=&\langle a,f\rangle,
\end{eqnarray*}
for every $a\in\mathfrak{A}$. Therefore, $a\circ\pi^{\circ**}M=a$ for every $a\in\mathfrak{A}$. Thus $\mathfrak{A}$ has a symmetric virtual diagonal and consequently, Theorem \ref{t1} implies that $\mathfrak{A}$ is symmetrically amenable.
\end{proof}


\end{document}